\documentclass[12pt,reqno]{amsart}


\author{Chantal David}
\address{Concordia University \\Department of Mathematics and Statistics \\ 1455 de Maisonneuve West \\ Montr\'{e}al, Qu\'{e}bec H3G1M8 \\ Canada}
\email{cdavid@mathstat.concordia.ca}

\author{Derek Garton}
\address{Northwestern University \\ Department of Mathematics \\ 2033 Sheridan Road \\ Evanston, IL 60208 \\ United States}
\email{derek@math.northwestern.edu}

\author{Zachary Scherr}
\address{University of Michigan \\ Department of Mathematics \\  East Hall \\ Ann Arbor, MI 48109 \\ United States}
\email{zscherr@umich.edu}

\author{Arul Shankar}
\address{Institute for Advanced Study \\ School of Mathematics \\ Einstein Drive \\ Princeton, NJ 08540 \\ United States}
\email{arul.shnkr@gmail.com}

\author{Ethan Smith}
\address{Liberty University \\ Mathematics Department \\ 1971 University Boulevard \\ Lynchburg, VA 24502 \\ United States}
\email{ecsmith13@liberty.edu}

\author{Lola Thompson}
\address{University of Georgia\\ Department of Mathematics \\ Boyd Graduate Studies Building \\ Athens, GA 30601 \\ United States}
\email{lola@math.uga.edu}

\title{Abelian surfaces over finite fields with prescribed groups}
\usepackage{amsfonts,amsmath,amssymb,amsthm,geometry,enumerate}
\DeclareMathAlphabet{\curly}{U}{rsfs}{m}{n}

\newtheorem{thm}{Theorem}[section]
\newtheorem{prop}[thm]{Proposition}

\newtheorem{corollary}[thm]{Corollary}

\newtheorem{lemma}[thm]{Lemma}
\theoremstyle{remark}
\theoremstyle{definition}

\newtheorem*{rmk}{Remark}

\begin{document}
\def\phi{\varphi}
\renewcommand{\labelenumi}{(\roman{enumi})}
\def\polhk#1{\setbox0=\hbox{#1}{\ooalign{\hidewidth
    \lower1.5ex\hbox{`}\hidewidth\crcr\unhbox0}}}
\newcommand{\del}{\ensuremath{\delta}}
\def\A{\curly{A}}
\def\B{\curly{B}}
\def\e{\mathrm{e}}
\def\E{\curly{E}}
\def\F{\Bbb{F}}
\def\C{\mathbf{C}}
\def\I{\curly{I}}
\def\N{\mathbf{N}}
\def\D{\curly{D}}
\def\Q{\mathbb{Q}}
\def\O{\curly{O}}
\def\V{\curly{V}}
\def\W{\curly{W}}
\def\Z{\Bbb{Z}}
\def\p{\tilde{p}}
\def\Pp{\curly{P}}
\def\pr{\mathfrak{p}}
\def\Proj{\mathbf{P}}
\def\q{\mathfrak{q}}
\def\Ss{\curly{S}}
\def\T{\curly{T}}
\def\Nm{\mathcal{N}}
\def\cont{\mathrm{cont}}
\def\ord{\mathrm{ord}}
\def\rad{\mathrm{rad}}
\def\lcm{\mathop{\mathrm{lcm}}}
\newcommand{\lp}{\ensuremath{\left(}}
\newcommand{\rp}{\ensuremath{\right)}}

\numberwithin{equation}{section}
\begin{abstract}

Let $A$ be an abelian surface over $\F_q$, the field of $q$ elements.
The rational points on $A/\F_q$ form an abelian group
$A(\F_q) \simeq \Z/n_1\Z \times \Z/n_1 n_2 \Z \times \Z/n_1 n_2 n_3\Z  \times\Z/n_1 n_2 n_3 n_4\Z$.
We are interested in knowing which groups of this shape actually arise as the group of points on some abelian surface over some finite field.
For a fixed prime power $q$, a characterization of the abelian groups that occur was recently found by Rybakov.
One can use this characterization to obtain a set of congruences
%modulo the integers $n_1, n_2, n_3, n_4$
on certain combinations of coefficients of the corresponding Weil polynomials.
We use Rybakov's criterion to show that groups $\Z/n_1\Z \times \Z/n_1 n_2 \Z \times \Z/n_1 n_2 n_3\Z  \times\Z/n_1 n_2 n_3 n_4\Z$ do not occur if $n_1$ is very large with respect to $n_2, n_2, n_4$ (Theorem \ref{splitbound}), and occur with density zero in a wider
range of the variables  (Theorem \ref{splitbound-average}).
\end{abstract}
\maketitle

\markleft{C. David, D. Garton, Z. Scherr, A. Shankar, E. Smith, L. Thompson}

\section{Introduction}

Let $E$ be an elliptic curve over the finite field $\F_p$.
 It is well-known that the points on $E$ over $\F_p$ form a finite abelian group $E(\F_p)$ of ``rank" at most $2$, i.e.,
 \begin{equation}
 \label{groupoccurs} E(\F_p) \simeq \Z/n_1\Z \times \Z/n_1 n_2 \Z,
 \end{equation}
 for some positive integers $n_1, n_2$.
 It is natural to ask which groups arise in this manner as $p$ runs through all primes and as $E$ runs through
all curves over $\F_p$.
 Let ${S}(N_1, N_2)$ be the set of pairs of integers  $n_1 \leq N_1, n_2 \leq N_2$ such that there exists a prime $p$ and a curve $E / \F_p$ for which~\eqref{groupoccurs} holds.
The problem of estimating the size of $S(N_1,N_2)$ was first considered by Banks, Pappalardi, and Shparlinski in \cite{BPS},
who gave precise conjectures and numerical evidence for this problem.
In particular, they conjectured that the
%completely split groups
%(when $n_2 = 1$) and
``very split" groups (when $n_1$ is very large compared to $n_2$) occur with density zero.
This was proven by Chandee, David, Koukoulopoulos, and Smith in \cite{CDKS}, who showed that
%In \cite{CDKS}, the authors proved that
\begin{equation*}
\#S(N_1, N_2) = \underline{o}(N_1 N_2)
\end{equation*}
when $N_1 \geq \exp{(N_2^{1/2+\varepsilon})}$, or equivalently, when $N_2 \leq (\log N_1)^{2 - \varepsilon}$.
%, as conjectured in \cite{BPS}.
Positive density results were also conjectured in~\cite{BPS} and proven in part in~\cite{CDKS}.

In this paper we examine analogous questions for abelian surfaces over finite fields.
Here and throughout, $q$ will denote the prime power $p^r$, and $A$ will denote an abelian surface over the finite field $\F_q$.
The points on $A$ over $\F_q$ possess the structure of a finite abelian group $A(\F_q)$ of rank at most $4$, i.e.,
\begin{eqnarray}
\label{groupoccurs-AV}
A(\F_q) \simeq \Z/n_1\Z \times \Z/n_1 n_2 \Z \times \Z/n_1 n_2 n_3\Z  \times\Z/n_1 n_2 n_3 n_4\Z
\end{eqnarray}
for some positive integers $n_1, n_2, n_3, n_4$.
For the sake of convenience, we will use the notation $G(n_1, n_2, n_3, n_4)$ to refer to  the group on the right hand side
of~\eqref{groupoccurs-AV}.
We then want to study which of the groups $G(n_1, n_2, n_3, n_4)$ actually occur when
we vary over all finite fields $\F_q$ and over all abelian surfaces $A/\F_q$.

%%By Tate-Honda theory, simple abelian varieties over $\F_q$ are in bijection with conjugacy classes of Weil numbers.
For fixed $q$, a characterization of the groups occurring as the group of points on a general abelian variety
was recently found by Rybakov \cite{rybakov1, rybakov2}.
Rybakov's elegant criterion relates the Newton polygon of the characteristic polynomial of the variety to the Hodge polygon of the group.
The work of Rybakov may be viewed as generalizing R\"uck's characterization for elliptic curves~\cite{RU} to abelian varieties of any dimension.
We give a detailed description of these results in Section \ref{background}.
%Lola: Added "the" to "the Newton polygon of the characteristic polynomial" above.

As with the case of elliptic curves, we expect that the ``very split" groups $G(n_1, n_2, n_3, n_4)$
(viz., when  $n_1, n_2$ are large with respect to $n_3, n_4$)
are less likely to occur. This is compatible with the general philosophy of
the Cohen-Lenstra heuristics, which predict that random abelian groups naturally
occur with probability inversely proportional to the size of their automorphism groups.
Note that the very split groups have many more automorphisms than the cyclic group of the same size.
In fact, Rybakov's criterion shows that whenever there is an abelian variety with $N$ points over $\F_q$,
the cyclic group of order $N$ will always occur.

We now state our main results.
We recall that an abelian variety is simple if it is not isogenous to a product of abelian varieties of lower dimension.
Our first result is that some groups never occur for simple abelian surfaces over $\F_q$.
In particular,  when $n_1$ is too large with respect to $n_2, n_3, n_4$, the group $G(n_1,n_2,n_3,n_4)$ does not arise as the group
of points on any simple abelian variety over any finite field.
This is different from the case of elliptic curves, where such a statement is true only in a probabilistic sense,
viz., very split groups occur with density zero as proven in \cite{CDKS}.

\begin{thm} \label{splitbound}
Suppose that $n_1 ,n_2, n_3, n_4$ are positive integers. If $$n_1 \ge  60 n_2^{1/4} n_3^{3/2}  n_4^{3/4} + 1,$$
then for every $q$, there is no simple abelian surface
$A/\F_q$ with $A(\F_q)\simeq G(n_1,n_2,n_3,n_4)$.
\end{thm}

We also show that fewer groups occur in a probabilistic sense.
Our next result essentially says that if $n_1$ or $n_2$ is very large compared to $n_3$ and $n_4$,
then $G(n_1,n_2,n_3,n_4)$ occurs with probability zero.
%In particular, if $n_1, n_2$ are very large compared with $n_3, n_4$, the groups $G(n_1, n_2, n_3, n_4)$ occur with density zero.
Given $N_1,N_2,N_3,N_4\ge 1$, we define $S(N_1,N_2,N_3,N_4)$ to be the set of quadruples $(n_1,n_2,n_3,n_4)$ for which
$N_j\le n_j\le 2N_j$ for $1\le j\le 4$ and there exists a prime power $q$ and a simple abelian surface $A/\F_q$ with
$A(\F_q)\simeq G(n_1,n_2,n_3,n_4)$.
%Theorem~\ref{splitbound} of course implies that $S(N_1,N_2,N_3,N_4)$ will be empty if $N_1$ is large enough with respect to the other parameters.
Throughout, we write $f=\underline o(g)$ as $x\rightarrow\infty$ if $f/g\rightarrow 0$ as $x\rightarrow\infty$.

\begin{thm} \label{splitbound-average}
If
\begin{equation*}
\frac{N_1N_2^{1/4}}{N_3^{1/2}N_4^{1/4}}\rightarrow\infty
\end{equation*}
as $N_2N_4\rightarrow\infty$, then
\begin{equation*}
\#S(N_1,N_2,N_3,N_4)=\underline o(N_1N_2N_3N_4)
\end{equation*}
as $N_2N_4\rightarrow\infty$.
\end{thm}

\section{Weil polynomials and groups of abelian surfaces} \label{background}

%%Lola: Removed "by" in the sentence "This classification can be stated by using the characteristic polynomial..." below.

A classification of simple abelian varieties over $\F_q$ (up to $\F_q$-isogeny) is given by
Tate-Honda theory, which gives a one-to-one correspondence between isogeny classes of
simple abelian varieties over $\F_q$ and conjugacy classes of Weil numbers (algebraic integers whose
conjugates have absolute value $q^{1/2}$).
This classification can be stated using the characteristic polynomial of the Frobenius endomorphism $\pi_A$ of $A/\F_q$.
This polynomial, which we denote by $f_A(T)$, determines $A$ up to isogeny, and it has Weil numbers as its roots.
For an abelian surface $A/\F_q$, we write
\begin{equation*}
f_A(T) = T^4 + a_1 T^3 + a_2 T^2 + a_1 q T + q^2.
\end{equation*}
The number of $\F_q$-rational points on $A$ is equal to $f_A(1)$ and hence is an invariant of the isogeny class.
The fact that the roots of $f_A(T)$ are Weil numbers implies that
\begin{equation}\label{hasse-weil}
(\sqrt q -1)^4\le\#A(\F_q)\le(\sqrt q+1)^4.
\end{equation}

If $A$ is a simple abelian surface, then $f_A(T) = h_A(T)^e$
where $h_A(T)$ is an irreducible polynomial in $\Z[T]$ whose roots are Weil numbers.
Furthermore, the endomorphism algebra $\mathrm{End}_{\F_q}(A) \otimes \Q$ is a field if and only if $e=1$. Computing the local invariants of the algebra $\mathrm{End}_{\F_q}(A) \otimes \Q$ allows one to obtain a correspondence between the set of simple abelian surfaces over $\F_q$ such that $\mathrm{End}_{\F_q}(A) \otimes \Q$ is a field and the set of irreducible polynomials $f(T)$ of degree 4 whose roots are Weil numbers and whose monic irreducible divisors $f_i(T)$ over $\Q_p$ have integer values of $\nu_p(f_i(0))/\nu_p(q)$. Here and throughout, we use the notation $\nu_p$ to denote the usual $p$-adic valuation. R\"{u}ck \cite{ruck} gave the following explicit characterization of these polynomials.

%Lola: I re-wrote the sentence in the paragraph above that starts with "Computing the local invariants..." because I felt that it was confusing. It used to say "By computing the
%local invariants of the algebra $\mathrm{End}_{\F_q}(A) \otimes \Q$, we have that
%the set of  simple abelian surfaces over $\F_q$ such that $\mathrm{End}_{\F_q}(A) \otimes \Q$
%is a field corresponds to the
%set of irreducible polynomials $f(T)$ of degree 4 whose roots are Weil numbers and such that
%for each irreducible monic divisor $f_i(T)$ of $f(T)$ over $\Q_p$, we have that $\nu_p(f_i(0))/\nu_p(q)$
%is an integer."

\begin{thm}[R\"{u}ck]\label{ruckthm} The set of $f_A(T)$ for all abelian varieties $A$ over $\F_q$ of dimension $2$ whose algebra $\mathrm{End}_{\F_q}{(A)}\otimes\Q$ is a field is equal to the set of polynomials $f(T) = T^4 + a_1T^3 + a_2T^2 + a_1qT + q^2$ where the integers $a_1$ and $a_2$ satisfy the conditions
\begin{enumerate}[(a.)]
\item $|a_1| < 4q^{1/2}$, $2|a_1|q^{1/2} - 2q < a_2 < a_1^2/4 + 2q$,
\item $a_1^2 - 4a_2 + 8q$ is not a square in $\Z$, and
\item either
\begin{enumerate}[(i.)]
\item $\nu_p(a_1) = 0, \nu_p(a_2) \geq r/2$ and $(a_2 + 2q)^2 - 4qa_1^2$ is not a square in $\Z_p$,
\item $\nu_p(a_2) = 0$, or
\item $\nu_p(a_1) \geq r/2$, $\nu_p(a_2) \geq r$, and $f(T)$ has no root in $\Z_p.$
\end{enumerate}
\end{enumerate}
\end{thm}

The polynomials $f_A(T)$ corresponding to simple abelian surfaces $A$ over $\F_q$
whose algebra $\mathrm{End}_{\F_q}(A) \otimes \Q$ is not a field are much rarer.
They can be described explicitly as well.

\begin{thm}[Waterhouse, Xing]\label{classification2}
The characteristic polynomial $f_A(T)$ of any simple abelian variety $A$ of dimension $2$ over $\F_q$
whose algebra $\mathrm{End}_{\F_q}{(A)}\otimes\Q$ is not field must be of the form
\begin{enumerate}[(a.)]
\item $f_A(T) = (T^2 - q)^2$ and $r$ is odd,
\item $f_A(T) = (T^2 + q)^2$, $r$ is even, and $p \equiv 1 \pmod 4$, or
\item $f_A(T) = (T^2 \pm q^{1/2} T + q)^2$, $r$ is even, and $p \equiv 1 \pmod 3$.
\end{enumerate}
\end{thm}

%%Lola: I combined the two sentences below. It used to say "See~\cite{xing2} also." It seemed unnecessary to keep them separate.

The group structures for these ``exceptional" polynomials $f_A(T)$ were studied by Xing in~\cite{xing1} and~\cite{xing2}. In the respective cases (corresponding to Theorem~\ref{classification2}),
Xing showed that the group structures which arise are precisely
\begin{enumerate}[{\it (a.)}]
\item $\left( \Z / (q-1) \Z \right)^2$, $\left( \Z / 2 \Z \right)^2\times \left( \Z / \frac{q-1}{2} \Z \right)^2 $, or $ \Z/2\Z \times \Z /\frac{q-1}{2} \Z\times\Z/(q-1)\Z$;
\label{casea}
\item $\left( \Z / (q+1) \Z \right)^2$; or
\item $\left( \Z / (q \pm q^{1/2} + 1) \Z \right)^2$.
\end{enumerate}

We refer the reader to~\cite{xing1} for a precise description of when each group corresponding to the first case arises.
%\begin{thm}[Xing]\label{xingthm} If $r$ is odd and $f_A(T) = (T^2 - q)^2$, then $A(\F_q)$
%is one of the following three groups:
%\begin{eqnarray*}
% \left( \Z / (q-1) \Z \right)^2, \quad
%\left( \Z / \frac{q-1}{2} \Z \right)^2 \times \left( \Z / 2 \Z \right)^2,
%\quad  \Z / (q-1) \Z \times
% \Z / \frac{q-1}{2} \Z  \times \Z / 2 \Z  .
%\end{eqnarray*}
%If $r$ is even,  $p \equiv 1 \mod 4$ and $f_A(T) = (T^2 + q)^2$, then
%\begin{eqnarray*}
%A(\F_q) \simeq \left( \Z / (q+1) \Z \right)^2.
%\end{eqnarray*}
%If  $r$ is even,  $p \equiv 1 \mod 3$, and $f_A(T) = (T^2 \pm q^{1/2} T + q)^2$, then
%\begin{eqnarray*}
%A(\F_q) \simeq \left( \Z / (q \pm q^{1/2} + 1) \Z \right)^2.
%\end{eqnarray*}
%\end{thm}
Thus, the abelian surfaces $A$ whose algebra $\mathrm{End}_{\F_q}{(A)}\otimes\Q$ is not a field give rise to very few groups
$G(n_1, n_2, n_3, n_4)$.
More importantly, $n_1,n_2\le 2$ for all such groups,and hence they
do not satisfy the conditions of Theorem \ref{splitbound} or Theorem \ref{splitbound-average}.
Therefore, we exclude this case from consideration for the remainder of the paper.

For the typical case of abelian surfaces whose algebra is a field, there is a very elegant criterion due to Rybakov \cite{rybakov1, rybakov2}
that characterizes those isogeny classes which contain a variety $A$ with $A(\F_q)\simeq G(n_1, n_2, n_3, n_4)$.
The result of Rybakov applies to abelian varieties of any dimension $g\ge 1$.
We state it below in full generality and then for the particular case of abelian surfaces.
We first need some definitions.

Let $\ell$ be a prime, and let $Q(T)=\sum_i Q_iT^i$ be a polynomial of degree $d$ with $Q(0)=Q_0\ne 0$.
The \textit{Newton polygon} $\mathrm{Np}_\ell(Q)$ is the boundary (without vertical lines) of the lower convex hull of the points
 $(i, \nu_\ell(Q_i))$ for $0 \leq i \leq d$ in $\mathbb{R}^2$.
 Now let $0 \leq m_1 \leq m_2 \leq \cdots \leq m_r$ be nonnegative integers, and let $H=\bigoplus_{i=1}^r \Z/\ell^{m_i}\Z$.
 The \textit{Hodge polygon} $\mathrm{Hp}_\ell(H,r)$ is the convex polygon with vertices $(i, \sum_{j=1}^{r-i} m_j)$ for $0 \leq i < r$.
 Given an abelian group $G$, we let $G_\ell$ denote the $\ell$-primary component of $G$.
 The following is the main result of~\cite{rybakov1}.

\begin{thm}[Rybakov]\label{thm-rybakov2}
Let $A$ be an abelian variety of dimension $g$ over a finite field whose algebra $\mathrm{End}_{\F_q}{(A)}\otimes\Q$ is a field.
Let $f_A(T)$ denote its characteristic polynomial,
and let $G$ be an abelian group of order $f_A(1)$ that can be generated by $2g$ or fewer elements.
Then $G$ is the group of points on some variety in the isogeny class of $A$ if and only if $\mathrm{Np}_\ell(f_A(1-T))$ lies on or above
$\mathrm{Hp}_\ell(G_\ell, 2g)$ for every prime number $\ell$.
\end{thm}

For the case of abelian surfaces, we rewrite the conditions of Theorem \ref{thm-rybakov2} explicitly as follows.

\begin{corollary}\label{congs}
Let $A/\F_q$ be an abelian surface, and suppose that $\mathrm{End}_{\F_q}{(A)}\otimes\Q$ is a field.
Let $f_A(T)= T^4 + a_1T^3 + a_2T^2 + a_1qT + q^2$ denote its Weil polynomial.
Then the isogeny class of $A$ contains a variety with group of points isomorphic to $G(n_1,n_2,n_3,n_4)$
if and only if
\begin{equation} \label{congone}
n_1^4n_2^3n_3^2n_4 = f_A(1) = q^2+a_1q+a_2+a_1+1
\end{equation}
 and
\begin{eqnarray}
4+3a_1+2a_2+qa_1&\equiv& 0\pmod{n_1^3n_2^2n_3},\label{congtwo}\\
6+3a_1+a_2&\equiv& 0\pmod{n_1^2n_2},\label{congthree}\\
4+a_1&\equiv& 0\pmod{n_1}.\label{congfour}
\end{eqnarray}
\end{corollary}

%%Lola: The corollary previously defined $f_A(t) = T^4 +...$. I changed it to $f_A(T)$ since our parameter is capital T, not lowercase t.

We remark that Corollary \ref{congs} implies that if $f_A(1)=N$, then the cyclic group of order $N$ occurs as a group of points on some abelian surface in the isogeny class of $A$ since in that case we have $n_1 = n_2 = n_3=1$, so the congruences \eqref{congtwo} -- \eqref{congfour} are  trivially satisfied.

\section{Key proposition}

In this section we prove the following key proposition.
As is somewhat common, for any real number $x$, we write $||x||$ for the distance between $x$ and its nearest integer neighbor.
To ease notation, we define
\begin{equation}\label{delta}
\delta
=\delta(n_1,n_2,n_3,n_4)
=\begin{cases}
1&\text{if } 2n_3\sqrt{n_2n_4}\in\Z,\\
\| 2n_3\sqrt{n_2n_4} \|&\text{otherwise}
\end{cases}
\end{equation}
for any positive integers $n_1,n_2,n_3,n_4$.

\begin{prop}\label{key}
Suppose that $A/\F_q$ is a simple abelian surface and that $\mathrm{End}_{\F_q}(A)\otimes\Q$ is a field.
Suppose further that $A(\F_q)\simeq G(n_1,n_2,n_3,n_4)$.
Then
\begin{equation*}
n_1<\frac{10n_3^{1/2}n_4^{1/4}}{\delta n_2^{1/4}}+\frac{1}{n_2^{3/4}n_3^{1/2}n_4^{1/4}}.
\end{equation*}
\end{prop}

%Lola: Changed the wording in the next two paragraphs to make the sentences seem less abrupt.

Theorems \ref{splitbound} and \ref{splitbound-average} will follow from Proposition \ref{key}.
Theorem~\ref{splitbound} follows by specifying lower bounds for $\| \sqrt{m} \|$ that are valid for every integer $m$.
For the details, see Section~\ref{DNE}.
Theorem~\ref{splitbound-average} follows from the fact that the triple sequence $2n_3\sqrt{n_2n_4}$ is uniformly distributed modulo one; we will prove this in Section~\ref{UD}.

First, we use the congruences of Corollary~\ref{congs} to derive a simpler congruence on $a_1$.
\begin{lemma}\label{derekcongs}
Suppose that $q$, $a_1$, $a_2$, $n_1$, $n_2$, $n_3$, and $n_4$ satisfy~\eqref{congone}--\eqref{congfour}.
Then
\begin{equation*}
a_1 \equiv -2(q+1)\pmod{n_1^2n_2}.
\end{equation*}
\end{lemma}

\begin{proof}
Reducing \eqref{congone} and \eqref{congtwo} modulo $n_1^3n_2^2$ yields
\begin{align}
 q^2+a_1q+a_2+a_1+1&\equiv 0\pmod{n_1^3n_2^2}\label{reduced1},\\
4+3a_1+2a_2+qa_1&\equiv 0\pmod{n_1^3n_2^2}\label{reduced2}.
\end{align}
Reducing the above congruences modulo $n_1^2n_2$ and taking their difference gives
\begin{equation*}
2a_1+a_2+3-q^2\equiv0\pmod{n_1^2n_2}.
\end{equation*}
Subtracting this from \eqref{congthree}, we obtain
\begin{equation}
a_1+3+q^2\equiv 0\pmod{n_1^2n_2}\label{congsix}.
\end{equation}
The remainder of the proof is devoted to showing that $n_1^2n_2\mid (q-1)^2$,
which together with~\eqref{congsix} implies the desired congruence $a_1\equiv -2(q+1)\pmod{n_1^2n_2}$.

Taking twice \eqref{reduced1} and subtracting off \eqref{reduced2} yields
\begin{equation}\label{congnine}
2q^2+a_1q-a_1-2\equiv 0\pmod{n_1^3n_2^2}.
\end{equation}
From \eqref{congsix} we know that there is an integer $k$ such that $a_1=-3-q^2+kn_1^2n_2$.
After some slight rearrangement, plugging this expression for $a_1$ into~\eqref{congnine} gives
\begin{equation}
kn_1^2n_2(q-1)-(q-1)^3\equiv 0\pmod{n_1^3n_2^2}
\label{congeleven}.
\end{equation}

Working prime by prime, we will show that~\eqref{congeleven} implies that $n_1^2n_2\mid (q-1)^2$.
To this end let $\ell$ be an arbitrary prime, and suppose that $\nu_\ell(n_1^2n_2)=r$.
Then we want to show that $\nu_\ell((q-1)^2)\ge r$.
Assume for the sake of contradiction that $\nu_\ell((q-1)^2)<r$.
Since $\nu_\ell(kn_1^2n_2)\ge r$, it follows that $\nu_\ell(kn_1^2n_2-(q-1)^2)=\nu_\ell((q-1)^2)$, and hence
\begin{equation*}
\nu_\ell(kn_1^2n_2(q-1)-(q-1)^3)
=\nu_\ell(q-1)+\nu_\ell(kn_1^2n_2-(q-1)^2)
=3\nu_\ell(q-1)
<\frac{3r}{2}.
\end{equation*}
On the other hand since $n_1^3n_2^2$ divides $kn_1^2n_2(q-1)-(q-1)^3$, it follows that
\begin{equation*}
3\nu_\ell(n_1)+2\nu_\ell(n_2)\le\nu_\ell(kn_1^2n_2(q-1)-(q-1)^3)<\frac{3r}{2}=\frac{3}{2}(2\nu_\ell(n_1)+\nu_\ell(n_2)).
\end{equation*}
However, this implies that $\nu_\ell(n_2)<0$, which is impossible since $n_2$ is an integer.
\end{proof}

\begin{lemma}\label{k-interval lemma}
Suppose that $A/\F_q$ is a simple abelian surface, $\mathrm{End}_{\F_q}(A)\otimes\Q$ is a field, and $A(\F_q)\simeq G(n_1,n_2,n_3,n_4)$.
If $f_A(T) =T^4+a_1T^3+a_2T^2+qa_1T+q^2$ is the characteristic polynomial of $A/\F_q$, then there exists an integer $k$ such that
\begin{equation*}
a_1=kn_1^2n_2-2(q+1)
\end{equation*}
and
\begin{equation}\label{k-interval}
 2n_3\sqrt{n_2n_4}\left(\frac{\sqrt{q}-1}{\sqrt{q}+1}\right)^2<k < 2n_3\sqrt{n_2n_4}\left(\frac{\sqrt{q}+1}{\sqrt{q}-1}\right)^2.
 \end{equation}
\end{lemma}
\begin{proof}
By Theorem~\ref{ruckthm}(a), we know that $-4\sqrt{q} < a_1 < 4\sqrt{q}$.
By Lemma~\ref{derekcongs}, there exists an integer $k$ such that $a_1=n_1^2n_2k-2(q+1)$.
Substituting this into the bounds for $a_1$ and adding $2(q+1)$ to each side of the inequalities yields
\begin{equation*}
2q-4\sqrt{q} +2< n_1^2n_2k < 2q+4\sqrt{q}+2.
\end{equation*}
Factoring and dividing through by $n_1^2n_2$ allows us to obtain
\begin{equation*}
\frac{2(\sqrt{q} -1)^2}{n_1^2n_2}<k <\frac{2(\sqrt{q}+1)^2}{n_1^2n_2}.
\end{equation*}
Since $\#A(\F_q) = n_1^4n_2^3n_3^2n_4$, the Weil bound~\eqref{hasse-weil} implies that
\begin{equation*}
\frac{(\sqrt{q}-1)^2}{n_3\sqrt{n_2n_4}}\le n_1^2n_2\le\frac{(\sqrt{q}+1)^2}{n_3\sqrt{n_2n_4}}.
\end{equation*}
Together these bounds imply~\eqref{k-interval}.
%\begin{equation*}
% 2n_3\sqrt{n_2n_4}\left(\frac{\sqrt{q}-1}{\sqrt{q}+1}\right)^2<k < 2n_3\sqrt{n_2n_4}\left(\frac{\sqrt{q}+1}{\sqrt{q}-1}\right)^2.
%\end{equation*}
\end{proof}

%%Lola: Changed some wording in the above proof to avoid repetition and add precision.

For $q$ large enough, the interval from Lemma~\ref{k-interval lemma} will contain at most one integer $k$.
The following lemma makes this statement precise.
Recall the definition of $\delta$ given by~\eqref{delta}.

%%Lola: Changed some wording in the two previous sentences to make the statements more precise.

%%Lola: Changed wording in the first sentence below to avoid saying "integer" so many times.

\begin{lemma}\label{shrink k-interval}
If $\sqrt q\ge 10n_3\sqrt{n_2n_4}/\delta$, then the interval~\eqref{k-interval} contains no integral values of $k$ unless
$2n_3\sqrt{n_2n_4}$ is an integer, in which case $k=2n_3\sqrt{n_2n_4}$.
\end{lemma}
\begin{proof}
To further ease notation, let $m=2n_3\sqrt{n_2n_4}$.
Note that the interval $(m-\delta,m+\delta)$ does not contain an integer unless $m=2n_3\sqrt{n_2n_4}$ is itself an integer,
in which case it is the only such integer.
Since
\begin{equation*}
\left(\frac{\sqrt{q}+1}{\sqrt{q}-1}\right)^2=1+\frac{4\sqrt q}{(\sqrt q- 1)^2}
\quad\text{and}\quad
\left(\frac{\sqrt{q}-1}{\sqrt{q}+1}\right)^2=1-\frac{4\sqrt q}{(\sqrt q+ 1)^2},
\end{equation*}
%and
%\begin{equation*}
%\left(\frac{\sqrt{q}-1}{\sqrt{q}+1}\right)^2=1-\frac{4\sqrt q}{(\sqrt q+ 1)^2},
%\end{equation*}
it follows that the interval~\eqref{k-interval} is contained in the interval $(m-\delta,m+\delta)$ if and only if
\begin{equation*}
m\frac{4\sqrt q}{(\sqrt q -1)^2}\le \delta.
\end{equation*}
Factoring the latter inequality and dividing by $\delta$ yields
\begin{equation}\label{quadratic}
0\le \left(\sqrt q-\frac{2m+\delta-2\sqrt{m^2+m\delta}}{\delta}\right) \left(\sqrt q-\frac{2m+\delta+2\sqrt{m^2+m\delta}}{\delta}\right).
\end{equation}
Now, since
\[
\frac{2m+\delta-2\sqrt{m^2+m\delta}}{\delta}
\le1,
\]
it follows that~\eqref{quadratic} holds if and only if
\[
\sqrt{q}\ge\frac{2m+\delta+2\sqrt{m^2+m\delta}}{\delta}.
\]
However,
\begin{equation*}
\frac{2m+\delta+2\sqrt{m^2+m\delta}}{\delta}
\le\frac{2m+1+2\sqrt{m^2+m+1/4}}{\delta}
=\frac{4m+2}{\delta}
\le\frac{5m}{\delta},
\end{equation*}
and so~\eqref{quadratic} holds if $\sqrt q\ge 5m/\delta=10n_3\sqrt{n_2n_4}/\delta$.
\end{proof}

\begin{proof}[Proof of Proposition \ref{key}]
First, suppose that $k=2n_3\sqrt{n_2n_4}$ is an integer and
\begin{equation*}
a_1=kn_1^2n_2-2(q+1)=2n_1^2n_2^{3/2}n_3n_4^{1/2}-2(q+1).
\end{equation*}
Then substitution into~\eqref{congone} gives
\begin{equation*}
a_2=n_1^4n_2^3n_3^2n_4-1-\left(2n_1^2n_2^{3/2}n_3n_4^{1/2}-2(q+1)\right)(q+1)-q^2.
\end{equation*}
Under these assumptions, we then find that
\begin{equation*}
a_1^2-4a_2+8q=0.
\end{equation*}
According to Theorem~\ref{ruckthm}(b), this contradicts the assumption that $\mathrm{End}_{\F_q}\otimes\Q$ is field.
Therefore, by Lemmas~\ref{k-interval lemma} and~\ref{shrink k-interval},
regardless of whether $2n_3\sqrt{n_2n_4}$ is an integer, we see that if $A(\F_q)\simeq G(n_1,n_2,n_3,n_4)$ and
$\mathrm{End}_{\F_q}\otimes\Q$ is a field, then
$\sqrt q< 10n_3\sqrt{n_2n_4}/\delta$.
Using this together with the Weil bound~\eqref{hasse-weil}, we have that
\begin{equation*}
n_1 n_2^{3/4} n_3^{2/4} n_4^{1/4}\le \sqrt{q}+1<10n_3\sqrt{n_2n_4}/\delta+1.
\end{equation*}
Whence,
\begin{equation*}
n_1<\frac{10n_3^{1/2}n_4^{1/4}}{\delta n_2^{1/4}}+\frac{1}{n_2^{3/4}n_3^{1/2}n_4^{1/4}}.
\end{equation*}
\end{proof}

%%Lola: In the paragraph above, I changed "whether... is an integer or not" to "regardless of whether" to simplify the sentence.

\section{Proof of Theorem~\ref{splitbound}}\label{DNE}

The proof of Theorem \ref{splitbound} can be deduced from the following simple observation, which gives a lower bound for
$\| 2 n_3\sqrt{n_2n_4} \|$.
As usual, for any real number $x$, we write $[x]$ for the largest integer
less than or equal to $x$, and $\{ x \} = x - [x]$ for the fractional part of $x$.

\begin{lemma}\label{squareroot}
  Let $m$ be an integer that is not a perfect square. Then
  $$\| \sqrt{m} \| >
  \frac{1}{3 \sqrt{m}}.$$
\end{lemma}
\begin{proof}
Since $\sqrt m=[\sqrt m]+\{\sqrt m\}$,
upon squaring both sides, we find that
\begin{equation*}
\begin{split}
m
&=[\sqrt m]^2+2\{\sqrt m\}[\sqrt m]+\{\sqrt m\}^2\\
%&=[\sqrt m]^2+(2[\sqrt m]+\{\sqrt m\})\{\sqrt m\}\\
&=[\sqrt m]^2+(\sqrt m+[\sqrt m])\{\sqrt m\}\\
&\le[\sqrt m]^2+2\sqrt m\{\sqrt m\}.
\end{split}
\end{equation*}
Therefore, since $1 < m-[\sqrt m]^2$, we have $\{\sqrt m\}>1/2\sqrt m$.

Similarly, since we can write $\sqrt m=[\sqrt m]+1-(1-\{\sqrt m\})$, we have that
\begin{equation*}
\begin{split}
m
&=([\sqrt m]+1)^2-2(1-\{\sqrt m\})([\sqrt m]+1)+(1-\{\sqrt m\})^2\\
%&=([\sqrt m]+1)^2-(2[\sqrt m]+1+\{\sqrt m\})(1-\{\sqrt m\})\\
&=([\sqrt m]+1)^2-(\sqrt m+[\sqrt m]+1)(1-\{\sqrt m\})\\
&\ge([\sqrt m]+1)^2-(2\sqrt m+1)(1-\{\sqrt m\}).
\end{split}
\end{equation*}
Therefore, since $1\le ([\sqrt m]+1)^2-m$, we obtain $1-\{\sqrt m\}\ge1/(2\sqrt m+1)>1/3\sqrt m$.
\end{proof}

%%Reworded the above sentences so that the two cases don't look like they were copied-and-pasted from one another.

\begin{proof}[Proof of Theorem \ref{splitbound}]
By Proposition~\ref{key}, if there is a prime power $q$ and a simple abelian surface $A/\F_q$ with group $G(n_1, n_2, n_3, n_4)$, then
\begin{equation*}
 n_1<\frac{10n_3^{1/2}n_4^{1/4}}{\delta n_2^{1/4}}+\frac{1}{n_2^{3/4}n_3^{1/2}n_4^{1/4}}
\le\frac{10n_3^{1/2}n_4^{1/4}}{\delta n_2^{1/4}}+1,
\end{equation*}
where $\delta$ is as defined by~\eqref{delta}.
By Lemma \ref{squareroot},
\begin{eqnarray*}
n_1 < 60 n_2^{1/4} n_3^{3/2} n_4^{3/4} + 1
\end{eqnarray*}
since $\delta \ge (6 n_3 \sqrt{n_2 n_4})^{-1}$.
\end{proof}

\section{Proof of Theorem~\ref{splitbound-average}}\label{UD}

In this section we use the standard notation $f\ll g$ to mean that there exists a positive constant $c$ such that $|f|\le cg$.
We also use the notation $n\asymp N$ (in a somewhat nonstandard way) as a shorthand for $N\le n\le 2N$.
%We also write $f\asymp g$ to mean that both relations $f\ll g$ and $g\ll f$ hold.

To prove Theorem \ref{splitbound-average}, we use the fact that for most triples of integers $(n_2,n_3,n_4)$ with $n_j\asymp N_j$ ($2\le j\le 4$),
the distance between $2 n_3 \sqrt{n_2 n_4}$ and the nearest integer is larger than any function tending to zero as $N_2N_4\rightarrow\infty$.
This follows from the uniform distribution of $2 n_3 \sqrt{n_2 n_4}$ modulo one; see Theorem \ref{thm-uniform-dist-f} below.
For the sake of completeness, we review much of the relevant material here.

%%%Lola: Added semicolon to make "See Theorem \ref{{thm-uniform-dist-f} below" feel less abrupt.

Let
\begin{equation*}
 \mathcal{T}(N_2, N_3, N_4) = \left\{ (n_2, n_3, n_4) : n_2 \asymp N_2,  n_3 \asymp N_3, n_4 \asymp N_4\right\},
\end{equation*}
and let $\left\{ f(n_2,n_3,n_4): n_2, n_3, n_4\ge 1\right\}$
%$\left( f(n_2,n_3,n_4) \right)_{n_2, n_3, n_4\ge 1}$
 be any triply indexed sequence of real numbers.
For $0 \leq \alpha < \beta \leq 1$, let
\begin{eqnarray*}
Z_f(N_2,N_3,N_4; \alpha, \beta) = \# \left\{  (n_2, n_3, n_4) \in \mathcal{T}(N_2, N_3, N_4) : \alpha \leq \left\{ f(n_2,n_3,n_4) \right\} \leq \beta \right\},
\end{eqnarray*}
where, as in the previous section, $\left\{ f(n_2,n_3,n_4) \right\}$ denotes the fractional part of $f(n_2, n_3, n_4)$.
We say that the sequence $f(n_2,n_3,n_4)$  is uniformly distributed modulo one if
\begin{eqnarray*}
\lim_{N_2,N_3,N_4 \rightarrow \infty} \frac{Z_f(N_2,N_3, N_4; \alpha, \beta)}{N_2N_3N_4}  = \beta - \alpha.
\end{eqnarray*}
By Weyl's criterion, this is equivalent to showing that
\begin{eqnarray*}
E_k(N_2, N_3, N_4) :=
\sum_{\substack{n_2\asymp N_2,\\ n_3\asymp N_3,\\ n_4\asymp N_4}}
e \left( k f(n_2,n_3,n_4) \right) = \underline{o}
\left( N_2 N_3 N_4 \right)
\end{eqnarray*}
for every integer $k \neq 0$. As usual, we have written $e ( x ) = e^{2 \pi i x}.$
We can put this equivalence in quantitative form using the Selberg polynomials.
This is explained in Chapter 1 of \cite{Mont} for a sequence of one variable.
The proof for a sequence of three variables $f(n_2,n_3,n_4)$ follows along the same lines.
The next theorem is then the analogue of \cite[Chapter 1, Theorem 1]{Mont} for triple sequences.

%%Lola: Why do we call the next theorem a Theorem rather than a Lemma?

\begin{thm} \label{thm-montgomery}
Let $f(n_2,n_3,n_4)$ be a sequence of real numbers, and let $0 \leq \alpha \leq \beta \leq 1$.
Then
\begin{multline}
\left| Z_f(N_2,N_3, N_4; \alpha, \beta) - (\beta-\alpha) \;{\# \mathcal{T}(N_2, N_3, N_4)}  \right| \\
\label{eq-thm-montgomery}
\quad \leq
\frac{\# \mathcal{T}(N_2, N_3, N_4)}{K+1} + 2 \sum_{k=1}^K \left( \frac{1}{K+1} + \min{\left( \beta-\alpha, \frac{1}{\pi k} \right)}
\right)
\left| E_k(N_2,N_3,N_4) \right|
\end{multline}
for any positive integers $N_2, N_3, N_4$, and $K$.
\end{thm}

\begin{proof} For each positive integer $K$, let $$S_K^{+}(n) = \sum_{-K \leq k \leq K} \widehat{S}_K^{+}(k) \, e{(kn)}$$
be the Selberg polynomial upper bounding the characteristic function of $[\alpha, \beta]$ as defined in~\cite[p.~6]{Mont}.
%\cite[Equation (21+), Chapter 1]{Mont}.
Then
\begin{eqnarray*}
Z_f(N_2, N_3, N_4; \alpha, \beta) &\leq&
\sum_{{{n_2 \asymp N_2} \atop { n_3 \asymp N_3}} \atop {n_4 \asymp N_4}}
S_K^{+}(f(n_2, n_3, n_4)) \\
&=& \sum_{-K \leq k \leq K}  \widehat{S}_K^{+}(k) \sum_{{{n_2 \asymp N_2} \atop { n_3 \asymp N_3}} \atop {n_4 \asymp N_4}} e{\left( k f(n_2, n_3, n_4) \right)}.
\end{eqnarray*}
Now, since $$\widehat{S}_K^{+}(0) = \beta-\alpha + \frac{1}{K+1}$$
and $$E_0(N_2, N_3, N_4) = \# \mathcal{T}(N_2, N_3, N_4),$$
we have that
\begin{eqnarray*}
&&Z_f(N_2, N_3, N_4; \alpha, \beta) - (\beta-\alpha) \; \# \mathcal{T}(N_2, N_3, N_4) \\
&& \quad \quad \quad \quad \leq
\frac{\# \mathcal{T}(N_2, N_3, N_4)}{K+1} + \sum_{{-K \leq k \leq K}\atop{k\neq0}} \widehat{S}_K^{+}(k) E_k(N_2, N_3, N_4).
\end{eqnarray*}
It follows from properties of Selberg polynomials that
\begin{equation*}
|\widehat{S}_K^{+}(k)| \leq  \frac{1}{K+1} + \min{\left( \beta-\alpha, \frac{1}{\pi |k|} \right)}
\end{equation*}
for $0 < |k| \leq K$.  See \cite[p.~8]{Mont} for example.
Combining the inequalities from above, we have
\begin{eqnarray*}
&&Z_f(N_2, N_3, N_4; \alpha, \beta) - (\beta-\alpha) \; \# \mathcal{T}(N_2, N_3, N_4) \\
&& \quad \quad \quad \quad \leq
\frac{\# \mathcal{T}(N_2, N_3, N_4)}{K+1} + 2 \sum_{1 \leq k \leq K}
\left( \frac{1}{K+1} + \min{\left( \beta-\alpha, \frac{1}{\pi |k|} \right)} \right)
 \left| E_k(N_2, N_3, N_4) \right|.
\end{eqnarray*}
Using the Selberg polynomials $S^-_K(n)$ as defined in~\cite[p.~6]{Mont}, the other inequality follows, as does the theorem.
\end{proof}

%%Lola: Removed "the" in several places in the sentence "It follows from properties of Selberg polynomials that..." (see above)

For the remainder of the paper, we will specialize to the the sequence $f(n_2, n_3, n_4) = 2 n_3\sqrt{n_2n_4}$.
We now bound the sum appearing in Theorem \ref{thm-montgomery} to show that
the sequence $2n_3\sqrt{n_2n_4}$ is uniformly distributed modulo one.
%for the sequence $2 n_3\sqrt{n_2n_4}$, we have that
%\begin{eqnarray*}
%\lim_{N_2N_4 \rightarrow \infty} \frac{Z_f(N_2,N_3, N_4; \alpha, \beta)}{\# \mathcal{T}(N_2 N_3 N_4)}  = \beta -\alpha
%\end{eqnarray*}
%for all $0 \leq \alpha \leq \beta \leq 1$ and $N_3 \geq 1$.
In order to obtain our result without any conditions on the relative sizes of the parameters $N_2, N_3, N_4$, we bound the sum
appearing in (\ref{eq-thm-montgomery}) in two different ways (Lemmas \ref{firstbound} and \ref{secondbound}).
First, we use the following result from~\cite[p.~77]{Kr}.

%%%Lola: Changed "In order to obtain our result without any conditions on the relative size the parameters" to "...sizes of the parameters."

\begin{lemma}\label{maxsum}
Let $g(t)$ be a real, continuously differentiable function on the interval $[a,b]$, with $|g'(t)|\ge \lambda>0$, and let $N>0$.
Then
\begin{equation*}
\sum_{a\le n\le b} \min{\left\{ N, 1/\|g(n)\| \right\}} \ll (|g(b)-g(a)| + 1)\left(N + \frac{1}{\lambda}\log(b-a+2)\right).
\end{equation*}
\end{lemma}

\begin{lemma} \label{firstbound}
For every $\varepsilon > 0$ and $K\ge 1$,
\begin{equation*}
\sum_{k\le K}\frac{1}{k}\left|E_k(N_2,N_3,N_4)\right| \ll (N_2 N_4)^{1/2+\varepsilon} N_3 K + (N_2 N_4)^{1+\varepsilon} \log{2K}.
\end{equation*}
\end{lemma}
\begin{proof}
Let
\begin{equation*}
b_n = \sum_{n_2 \asymp N_2} \sum_{{n_4 \asymp N_4} \atop {n_2 n_4 = n}} 1,
\end{equation*}
and note that $b_n\ll n^{\varepsilon/2}$.
Recall the well-known bound
\begin{equation*}
 \sum_{n \asymp N} e( \alpha n) \ll \min \left\{ N, 1 / \| \alpha \| \right\}.
\end{equation*}
See~\cite[p.~199]{IK} for example.
Applying Lemma \ref{maxsum}, we have
\begin{equation*}
\begin{split}
\sum_{k\le K}\frac{1}{k}\left|E_k(N_2,N_3,N_4)\right|
&= \sum_{k \leq K} \frac{1}{k} \sum_{N_2N_4\le n \le 4N_2N_4} b_n \sum_{n_3 \asymp N_3} e(2kn^{1/2} n_3) \\
&\ll (N_2 N_4)^{\varepsilon/2} \sum_{k \leq K} \frac{1}{k} \sum_{N_2N_4\le n \le 4 N_2 N_4}
\min \left\{ N_3, \frac{1}{\|2 k n^{1/2} \|} \right\} \\
&\ll (N_2 N_4)^{\varepsilon/2} \sum_{k \leq K}
(N_2 N_4)^{1/2} \left( N_3 + k^{-1} (N_2 N_4)^{1/2} \log(2N_2 N_4) \right) \\
&\ll (N_2 N_4)^{\varepsilon}
\left( (N_2 N_4)^{1/2} N_3 K + (N_2 N_4) \log 2K \right).
\end{split}
\end{equation*}
\end{proof}

%%Lola: Got rid of "as" in the sentence below ("...consequence of the van der Corput method as found in...")

We now bound the same sum using the following consequence of the van der Corput method found in~\cite[p.~94]{Ten}.
\begin{lemma}\label{van der corput}
Let $g(t)$ be a twice continuously differentiable function on the interval $[a,b]$ such that
$|g''(t)|\asymp \lambda >0$.
Then
\begin{equation*}
\sum_{a\le n\le b}e(g(n))\ll (b-a+1)\lambda^{1/2}+\lambda^{-1/2}.
\end{equation*}
\end{lemma}

\begin{lemma} \label{secondbound}
For every $K\ge 1$,
\begin{equation*}
\sum_{k\le K}\frac{1}{k}\left|E_k(N_2,N_3,N_4)\right| \ll K^{1/2} N_3^{3/2}  (N_2 N_4)^{3/4}+N_3^{1/2}(N_2 N_4)^{3/4}.
\end{equation*}
\end{lemma}
\begin{proof}
We first apply Lemma~\ref{van der corput} with $g(t)=2kn_3\sqrt{n_2t}$,
noting that
%$g'(t)=kn_3\sqrt{n_2/t}$,
$|g''(t)|=kn_3\sqrt{n_2}/2t^{3/2}$. This yields
\begin{equation*}
\begin{split}
\sum_{k\le K}\frac{1}{k}\left|E_k(N_2,N_3,N_4)\right|
&\ll\sum_{k\le K}\frac{1}{k}\sum_{\substack{n_2\asymp N_2,\\ n_3\asymp N_3}}
	\left(N_4^{1/4}(kn_3\sqrt{n_2})^{1/2}+N_4^{3/4}(kn_3\sqrt{n_2})^{-1/2}\right)\\
&\ll N_4^{1/4}K^{1/2}N_3^{3/2}N_2^{5/4}+N_4^{3/4}N_3^{1/2}N_2^{3/4}.
\end{split}
\end{equation*}
Then, applying Lemma~\ref{van der corput} again with $g(t)=2kn_3\sqrt{n_4t}$,
we see that the same bound holds with the roles of $N_2$ and $N_4$ reversed.
Therefore, we have
\begin{equation*}
\begin{split}
\sum_{k\le K}\frac{1}{k}\left|E_k(N_2,N_3,N_4)\right|
&\ll K^{1/2}N_3^{3/2}\min\{N_2^{5/4}N_4^{1/4}, N_2^{1/4}N_4^{5/4}\}+N_3^{1/2}(N_2N_4)^{3/4}\\
&\ll K^{1/2} N_3^{3/2}  (N_2 N_4)^{3/4}+N_3^{1/2}(N_2 N_4)^{3/4}.
\end{split}
\end{equation*}
\end{proof}

%%Lola: Added "This yields" before the displayed equations above. Otherwise, we were starting the sentence with math symbols!

%%Lola: Added "see" to the sentence "we see that..." (above)

Combining Lemmas~\ref{firstbound} and~\ref{secondbound},
we now show that the triple sequence $2n_3\sqrt{n_2n_4}$ is uniformly distributed modulo one.

\begin{thm} \label{thm-uniform-dist-f}
Let $N_2, N_3, N_4 \geq 1$, and $0 \leq \alpha <  \beta \leq 1$.
Then
\begin{equation*}
\lim_{N_2 N_4 \rightarrow \infty} \frac{Z_f(N_2, N_3, N_4; \alpha, \beta)}{N_2N_3N_4} = \beta - \alpha.
\end{equation*}
%Furthermore, if
%$\beta-\alpha=(N_2N_3N_4)^{-\varepsilon}$,
%then
%\begin{eqnarray*}
%Z_f(N_2,N_3, N_4; \alpha, \beta) &=& \underline{o}(N_2 N_3 N_4)
%\end{eqnarray*}
%as $N_2 N_4 \rightarrow \infty$.
\end{thm}
\begin{rmk}
Note that we do not require that each of $N_2,N_3$ and $N_4$ tends to infinity in the above limit.
Rather, we only require that the product $N_2N_4\rightarrow\infty$.
\end{rmk}
\begin{proof}
Fix $0<\varepsilon<1/16$.
Applying Lemmas~\ref{firstbound} and~\ref{secondbound} with $K=(N_2N_4)^{1/4}$, we see that
\begin{equation*}
\begin{split}
\sum_{k\le K}\frac{1}{k}\left|E_k(N_2,N_3,N_4)\right|
&\ll(N_2N_4)^{3/4+\varepsilon}N_3+\min\{(N_2N_4)^{1+\varepsilon},N_3^{3/2}(N_2N_4)^{7/8}\}\\
&\ll (N_2N_4)^{3/4+\varepsilon}N_3+ (N_2N_4)^{15/16+\varepsilon}N_3^{3/4}.
\end{split}
\end{equation*}
Since (with this same choice of $K$) we have $N_2N_3N_4/K=(N_2N_4)^{3/4}N_3$,
using the above bound in Theorem~\ref{thm-montgomery} yields the theorem.
%the first assertion of the theorem.
%The second assertion then follows immediately from the first.
\end{proof}

%%Lola: In the proof above, we were using \epsilon instead of \varepsilon. I don't care which we use, but we should be consisted. Elsewhere in the paper, \varepsilon is used.

\begin{proof}[Proof of Theorem \ref{splitbound-average}]
Let $F(N_2,N_4)$ be any function tending to infinity with $N_2N_4$ and satisfying the bound
\begin{equation}\label{f bound}
F(N_2,N_4)\le\frac{N_1N_2^{1/4}}{18N_3^{1/2}N_4^{1/4}}.
\end{equation}
%(Since $F(N_2,N_4)\rightarrow\infty$ as $N_2N_4\rightarrow\infty$, we have that $1/F(N_2,N_4)=o(1)$ as $N_2N_4\rightarrow\infty$.)
Without loss of generality, we may assume that $N_2N_4$ is large enough that $F(N_2,N_4)\ge 1$.
Hence, we may write
\begin{equation*}
\#S(N_1, N_2, N_3, N_4)=\#S_1(N_1, N_2, N_3, N_4)  + \#S_2(N_1, N_2, N_3, N_4),
\end{equation*}
where
\begin{eqnarray*}
S_1(N_1, N_2, N_3, N_4)
&:=& \left\{ (n_1, n_2, n_3, n_4) \in {S}(N_1, N_2, N_3, N_4): \| 2 n_3\sqrt{n_2n_4} \| \le 1/F(N_2,N_4) \right\},\\
S_2(N_1, N_2, N_3, N_4)
&:=& \left\{ (n_1, n_2, n_3, n_4) \in {S}(N_1, N_2, N_3, N_4): \| 2 n_3\sqrt{n_2n_4} \| > 1/F(N_2,N_4) \right\}.
\end{eqnarray*}
It follows from Theorem 5.6 that $\#S_1(N_1, N_2, N_3, N_4)=\underline o(N_1 N_2 N_3 N_4)$ as $N_2 N_4 \rightarrow \infty$.
On the other hand, if $(n_1, n_2, n_3, n_4) \in {S}_2(N_1, N_2, N_3, N_4)$, then by Proposition 3.1
%(and the ranges specified by the definition of $S(N_1,N_2,N_3,N_4)$),
\begin{equation*}\label{calc bound}
\begin{split}
N_1
\le n_1
&<\frac{10n_3^{1/2}n_4^{1/4}}{||2n_3\sqrt{n_2n_4}|| n_2^{1/4}}+\frac{1}{n_2^{3/4}n_3^{1/2}n_4^{1/4}}\\
%&\le\frac{10(2N_3)^{1/2}(2N_4)^{1/4}}{||2n_3\sqrt{n_2n_4}||N_2^{1/4}}+\frac{1}{N_2^{3/4}N_3^{1/2}N_4^{1/4}}\\
&<\frac{10(2N_3)^{1/2}(2N_4)^{1/4}}{1/F(N_2,N_4)N_2^{1/4}}+\frac{1}{N_2^{3/4}N_3^{1/2}N_4^{1/4}}\\
&< 18F(N_2,N_4)\frac{N_3^{1/2}N_4^{1/4}}{N_2^{1/4}}.
\end{split}
\end{equation*}
However, this contradicts our choice of $F(N_2,N_4)$ that satisfies~\eqref{f bound}.
Therefore, we conclude that $S_2(N_1,N_2,N_3,N_4)$ is empty,
and hence $S(N_1, N_2, N_3, N_4)=\underline o(N_1 N_2 N_3 N_4)$ as $N_2 N_4 \rightarrow \infty$.
\end{proof}

\subsection*{Acknowledgements}
We would like to thank J. Achter, P. Clark, E. Goren, A. Harper, A. Silverberg, and J. Wu for useful discussions
related to this work. The argument presented in Section 4, which allows us to obtain the result of Theorem 1.2
without any conditions on the relative size of the parameters $N_2, N_3, N_4$, was suggested to us
by J. Wu. We are very thankful for his gracious help.
We would also like to express our gratitude to the AMS Mathematics Research Communities 2012 program for
providing us with the opportunity to work together in a very pleasant working
environment and for its generous financial support.

%%Lola: Should we acknowledge Pete Clark? Ethan and I spoke with him a bit about the "every cyclic group appears" confusion.

\providecommand{\bysame}{\leavevmode\hbox
to3em{\hrulefill}\thinspace}
\providecommand{\MR}{\relax\ifhmode\unskip\space\fi MR }
\providecommand{\nMRhref}[2]{%
  \href{http://www.ams.org/mathscinet-getitem?mr=#1}{#2}
} \providecommand{\href}[2]{#2}

\end{document}